\documentclass[12pt,a4paper]{article}

\usepackage{amsmath, amsthm, amsfonts, amssymb,color,geometry,enumerate}
\usepackage{fancybox}
\usepackage{cases}
\usepackage{fancyhdr}
\usepackage[dvipdfmx]{graphicx}
\usepackage{amscd}

\theoremstyle{definition}

\newtheorem{theorem}{Theorem}[section]

\newtheorem{lemma}[theorem]{Lemma}
\newtheorem{definition}[theorem]{Definition}
\newtheorem{remark}[theorem]{Remark}

\newtheorem{example}[theorem]{Example}

\newtheorem{corollary}[theorem]{Corollary}

\makeatletter

\@addtoreset{equation}{section}
\makeatother

\renewcommand{\epsilon}{\varepsilon}    % epsilon 
\begin{document}

\title{{\bf Free infinite divisibility for generalized power distributions with free Poisson term}}
\author{\Large{Junki Morishita and Yuki Ueda}}
\date{}

\maketitle

\abstract{We study free infinite divisibility (FID) for a class which is called generalized power distributions with free Poisson term by using a complex analytic technique and a calculation for the free cumulants and Hankel determinants. In particular, our main result implies that (i) if $X$ follows the free Generalized Inverse Gaussian distribution, then $X^r$ follows an FID distribution when $|r|\ge1$, (ii) if $S$ follows the standard semicircle law and $u\ge 2$, then $(S+u)^r$ follows an FID distribution when $r\le -1$, and (iii) if $B_p$ follows the beta distribution with parameters $p$ and $3/2$, then (a) $B_p^r$ follows an FID distribution when $|r|\ge 1$ and $0<p\le 1/2$, and (b) $B_p^r$ follows an FID distribution when $r\le -1$ and $p>1/2$.}

%%%%%%%%%%%%%%%%%%%%%%%%%%%%%%%%%%%%%%%%%%%%%%%%%%%
%Section 1 
%%%%%%%%%%%%%%%%%%%%%%%%%%%%%%%%%%%%%%%%%%%%%%%%%%%

\section{Introduction}

In classical probability theory, many people studied the infinite divisible distributions as laws of L\'{e}vy processes and as a general class of important distributions (e.g. normal, Cauchy, gamma, t-Student distributions, etc) for probability theory, mathematical statistics and finance models. A probability measure $\mu$ on $\mathbb{R}$ is said to be {\it infinitely divisible} if for each $n\in\mathbb{N}$ there exists a probability measure $\rho_n$ on $\mathbb{R}$ such that $\mu=\rho_n^{\ast n}$ where the operation $\ast$ is the classical convolution. We recall that a probability measure $\mu$ on $\mathbb{R}$ is infinitely divisible if and only if its characteristic function $\hat{\mu}$ has the L\'{e}vy-Khintchine representation:
\begin{equation}
\hat{\mu}(u)=\exp\left( i\eta u-\frac{1}{2}au^2+\int_\mathbb{R}\left( e^{iut}-1-iut1_{[-1,1]}(t)\right)d\nu(t) \right), \qquad u\in \mathbb{R},
\end{equation}
where $\eta\in\mathbb{R}$, $a\ge 0$ (called the gaussian component) and $\nu$ is a L\'{e}vy measure on $\mathbb{R}$, that is, $\nu(\{0\})=0$ and $\int_\mathbb{R}(1\wedge t^2)d\nu(t)<\infty$. The triplet $(\eta,a,\nu)$ is uniquely determined and is called the characteristic triplet. For example, the normal distribution $N(m,\sigma^2)$ ($m\in\mathbb{R}$, $\sigma>0$) has the characteristic triplet $(m,\sigma^2,0)$, (so that it is infinitely divisible). In general, it is very difficult to find the characteristic triplet of probability measures, (in other words, to write the characteristic functions in the L\'{e}vy-Khintchine representation). Many people found and studied subclasses of probability measures as criteria for infinite divisibility. In particular, subclasses of infinitely divisible distributions were studied to understand infinitely divisible distributions which are preserved by powers, products and quotients of independent random variables. We say that a random variable $X$ follows the class of mixture of exponential distributions (ME) if $X$ is of the form $EZ$ where two random variables $E$, $Z$ are independent, $E$ follows an exponential distribution and $Z$ is positive. By the Goldie-Steutel theorem, the class ME is a subclass of infinitely divisible distributions (see \cite{G67, S67}). The class is important to understand infinitely divisible distributions which is preserved by powers and products of independent random variables, that is, if $X$ follows the class ME then so is $X^r$ when $r\ge 1$, and if $X$, $Y$ are independent and follow the class ME then so is $XY$. As one of the most important subclasses of infinitely divisible distributions, the class GGC (generalized gamma convolution) was introduced by Thorin (see \cite{T77a, T77b}). Bondesson (see \cite{B92}) proved that the subclass HCM (hyperbolically completely monotone) of the class GGC is preserved by powers (if $X$ follows the class HCM then so is $X^r$ if $|r|\ge 1$), products and quotients of independent random variables. Moreover Bondesson (see \cite{B15}) proved that if $X,Y$ follow the GGC and $X,Y$ are independent then so is $XY$. However, we are not sure about whether the class is preserved by powers of random variables.

In free probability theory, we have the corresponding problems for free infinitely divisible distributions which will be defined in Section 2.2. As one of the most important subclasses of free infinitely divisible distributions, we have the class of Free Regular distributions (FR) and the class of Univalent Inverse Cauchy transform (UI). Firstly, P\'{e}rez-Abreu and Sakuma (\cite{PS12}) introduced the class of free regular (FR) probability measures in terms Bercovici-Pata bijection. In \cite{AHS13}, the class FR was developed as a characterization of nonnegative free L\'{e}vy processes. Next, the class UI was introduced by Arizmendi and Hasebe (see \cite{AH13}) as a subclass of free infinitely divisible distributions. We will explain about the class UI in Section 2.3 We are not completely sure about whether the classes FR and UI are closed with respect to powers, products and quotients of free independent random variables. Hasebe (\cite{Has16}) studied the powers and products for the classes FR and UI by using complex analytic techniques. 

In this paper, we studied the class UI, specifically, we proved that if a random variable $X$ follows the class of the generalized power distributions with free Poisson term (GPFP)
\begin{equation}\label{X}
X\sim \frac{\sqrt{(b-x)(x-a)}}{x}\sum_{k=1}^N \frac{\alpha_k}{x^{l_k}} 1_{(a,b)}(x)dx,
\end{equation}
where $0<a<b$, $N\in\mathbb{N}$, $\alpha=(\alpha_1,\dots ,\alpha_N)\in (0,\infty)^N$ satisfying that the right hand side of \eqref{X} becomes a probability density function, and $l\in \mathbb{R}_<^N:=\{(l_1,\dots ,l_N)\in \mathbb{R}^N: l_1<\dots <l_N\}$, then (i) $X^r$ follows a UI distribution if $r\ge 1$ in the case when $l\in[t,t+1]_<^N$ for some $t\ge0$, and (ii) $X^r$ follows a UI distribution if $r\le-1$ in the case when $l\in[0,1]_<^N$. By the way, due to a calculation of Hankel determinants for the free cumulants, we proved that (iii) there exist five parameters $(a,b,N,\alpha, l)$ with $0<a<b$, $N\in\mathbb{N}$, $\alpha\in(0,\infty)^N$, $l\in [t,t+1]_<^N$ for some $t>0$ such that $X\sim \text{GPFP}(a,b,N,\alpha,l)$, but $X^{-1}$ does not follow an FID distribution and (iv) there exist five parameters with $ l_1<\dots <l_{k-1}\le l_1+1<l_k<\dots<l_N$, $N\ge2$ ($a,b,\alpha$ are usual parameters) for some $1<k\le N$, such that $X$ does not follow an FID distribution. The class of GPFP contains free Poisson distributions which have no atoms at $0$, free positive stable laws with index $1/2$, the class of free Generalized Inverse Gaussian distributions (fGIG), shifted semicircle laws and a partial class of beta distributions. From our results, we obtain some properties for powers of random variables which follow the fGIG distributions, shifted semicircle laws, and beta distributions.

This paper consists of 5 sections. In section 2, we introduce the free infinitely divisible distributions and the class UI. In section 3, we prove that probability measures whose pdf satisfied some assumptions are in the class UI by using complex analytic techniques. In section 4, we prove our results (i) and (ii) as above by using results of section 3 and give some examples. In section 5, we prove our results (iii) and (iv) as above by using the Hankel determinants for free cumulants.

%%%%%%%%%%%%%%%%%%%%%%%%%%%%%%%%%%%%%%%%%%%%%%%%%%%
%Section 2
%%%%%%%%%%%%%%%%%%%%%%%%%%%%%%%%%%%%%%%%%%%%%%%%%%%
\section{Preliminaries}

\subsection{Notations}
In this paper, we use the following notations.
\begin{itemize}
\item Let $X$ be a (non-commutative) random variable and $\mu$ a probability measure on $\mathbb{R}$. The notation $X\sim \mu$ means that $X$ follows $\mu$. 
\item The domain $\mathbb{C}^+$ (resp. $\mathbb{C}^-$) means the complex upper (resp. lower) half plane.
\item Let $I$ be an interval in $\mathbb{R}$. The notation $I\pm i0$ means the set $\{x\pm i0: x\in I\}$.
\item The complex function $z^p$ ($p\in\mathbb{R}$) means the principal value defined on $\mathbb{C}\setminus (-\infty,0]$. The complex function $\arg(z)$ means the argument of $z$ defined on $\mathbb{C}\setminus(-\infty,0]$ taking values in $(-\pi,\pi)$.
\item Let $N$ be a positive integer and $I$ an interval in $\mathbb{R}$. The set $I_<^N$ means a cone of $I^N$, that is, the set of all elements $(l_1,\dots ,l_N)\in I^N$ satisfying $l_1<\dots <l_N$.
\end{itemize}

\subsection{Analytic tools in free probability theory}
Let $\mu$ be a probability measure on $\mathbb{R}$. The {\it Cauchy transform} of $\mu$ is defined by
\begin{equation}\label{OriginalCauchy}
G_\mu(z):=\int_\mathbb{R} \frac{1}{z-x}d\mu(x), \qquad z\in \mathbb{C}\setminus \text{supp}(\mu).
\end{equation}
In particular, if $X\sim \mu$ then we sometimes write $G_X$ as the Cauchy transform of $\mu$. The function $G_\mu$ is analytic on the complex upper half plane $\mathbb{C}^+$. Bercovici and Voiculescu (\cite[Proposition 5.4]{BV93}) proved that for any $\gamma>0$ there exist $\lambda,M,\delta>0$ such that $G_\mu$ is univalent in the truncated cone
\begin{equation}
\Gamma_{\lambda,M}:=\{z\in \mathbb{C}^+: \lambda|\text{Re}(z)|<\text{Im}(z), \text{Im}(z)>M\},
\end{equation}
and the image $G_\mu(\Gamma_{\lambda,M})$ contains the triangular domain 
\begin{equation}
\Lambda_{\gamma,\delta}:=\{z\in\mathbb{C}^-: \gamma|\text{Re}(z)|< -\text{Im}(z), \text{Im}(z)>-\delta\}.
\end{equation} 
Therefore the right inverse function $G_\mu^{-1}$ exists on $\Lambda_{\gamma,\delta}$. We define the {\it Voiculescu transform} of $\mu$ by
\begin{equation}\label{Voi}
\phi_\mu(z):=G_\mu^{-1}\left(\frac{1}{z}\right)-z, \qquad \frac{1}{z}\in\Lambda_{\gamma,\delta} .
\end{equation}
It was proved in \cite{BV93} (historically, see also \cite{M92, V86}) that for any probability measures $\mu$ and $\nu$ on $\mathbb{R}$, there exists a unique probability measure $\lambda$ on $\mathbb{R}$ such that
\begin{equation}
\phi_\lambda(z)=\phi_\mu(z)+\phi_\nu(z),
\end{equation}
for all $z$ in the intersection of the domains of the three transforms. We write $\lambda:=\mu \boxplus \nu$ and call such $\lambda$ the {\it additive free convolution} (for short, {\it free convolution}) of $\mu$ and $\nu$. The operation $\boxplus$ is the free analogue of the classical convolution $\ast$. The {\it free cumulant transform} of $\mu$ is defined by
\begin{equation}
C_\mu(z):=z\phi_\mu\left(\frac{1}{z}\right), \qquad z\in \Lambda_{\gamma,\delta}.
\end{equation}
In particular, if $X\sim \mu$ then we sometimes write $C_X$ as the free cumulant transform of $\mu$. By definition, we have that $C_{\mu\boxplus\nu}(z)=C_\mu(z)+C_\nu(z)$ for all $z$ in the intersection of the domains of the three transforms. The transform is the free analogue of $\log\hat{\mu}$.

A probability measure $\mu$ on $\mathbb{R}$ is said to be {\it free infinitely divisible} if for each $n\in\mathbb{N}$ there exists a probability measure $\rho_n$ on $\mathbb{R}$ such that $\mu=\rho_n^{\boxplus n}$. Bercovici and Voiculescu (\cite[Theorem 5.10]{BV93}) proved that $\mu$ is free infinitely divisible if and only if the Voiculescu transform $\phi_\mu$ has analytic continuation to a map from $\mathbb{C}^+$ taking values in $\mathbb{C}^-\cup \mathbb{R}$, and therefore it has the {\it Pick-Nevanlinna representation}
\begin{equation}
\phi_\mu(z)= \gamma+\int_\mathbb{R} \frac{1+xz}{z-x}d\sigma(x), \qquad z\in\mathbb{C}^+,
\end{equation}
for some $\gamma\in\mathbb{R}$ and a nonnegative finite measure $\sigma$ on $\mathbb{R}$. The pair $(\gamma,\sigma)$ is uniquely determined and called the {\it free generating pair}. In this case, we can rewrite the free cumulant transform $C_\mu$ as
\begin{equation}\label{LKrep}
C_\mu(z)=\eta z+ az^2+ \int_\mathbb{R} \left(\frac{1}{1-tz}-1-tz 1_{[-1,1]}(t)\right) d\nu(t), \qquad z\in\mathbb{C}^-,
\end{equation}
where $\eta\in\mathbb{R}$, $a\ge 0$ (called the {\it semicircular component}) and $\nu$ is a nonnegative measure on $\mathbb{R}$ satisfying $\nu(\{0\})=0$ and $\int_\mathbb{R} (1\wedge t^2) d\nu(t)<\infty$ (see \cite[Proposition 5.2]{BT02}). The measure $\nu$ is called the {\it free L\'{e}vy measure} of $\mu$. The representation \eqref{LKrep} is called the {\it free L\'{e}vy-Khintchine representation} and it is the free analogue of classical L\'{e}vy-Khintchine representation of the Fourier transform $\hat{\mu}$. The triplet $(\eta, a,\nu)$ is uniquely determined by $\mu$ and called the {\it free characteristic triplet}. Moreover the relation between the free generating pair $(\gamma,\sigma)$ and the free characteristic triplet $(\eta,a,\nu)$ is given by
\begin{equation}\label{FC-FG}
\begin{split}
a&=\sigma(\{0\}),\\
d\nu(t)&=\frac{1+t^2}{t^2}\cdot 1_{\mathbb{R}\setminus \{0\}}(t) d\sigma(t),\\
\eta=&\gamma+\int_\mathbb{R}t\left(1_{[-1,1]}(t)-\frac{1}{1+t^2} \right)d\nu(t).
\end{split}
\end{equation}

\begin{example}\label{cumulanttrans}
(1) The semicircle distribution $S(m,\sigma^2)$ is a probability measure with the following probability density function:
\begin{equation}
\frac{1}{2\pi\sigma^2}\sqrt{4\sigma^2-(x-m)^2}1_{(m-2\sigma,m+2\sigma)}(x), \qquad m\in\mathbb{R},\sigma>0.
\end{equation}
It is known (for example, see \cite{NS06}) that
\begin{equation}
C_{S(m,\sigma^2)}(z)=mz+\sigma^2z^2, \qquad z\in\mathbb{C}^-.
\end{equation}
Therefore $S(m,\sigma^2)$ is free infinitely divisible and has the free characteristic triplet $(m,\sigma^2,0)$. This distribution appears as the limit of eigenvalue distributions of Wigner matrices as the size of the random matrices
goes to infinity. 

(2) The free Poisson distribution (or Marchenko-Pastur distribution) ${\bf fp}(p,\theta)$ is a probability measure given by
\begin{equation}
\max\{1-p,0\}\delta_0+\frac{\sqrt{\left(\theta(\sqrt{p}+1)^2-x\right)\left(x-\theta(\sqrt{p}-1)^2\right)}}{2\pi x}1_{(\theta(\sqrt{p}-1)^2,\theta(\sqrt{p}+1)^2)}(x)dx
\end{equation}
for $p,\theta>0$. It has a free cumulant transform
\begin{equation}
C_{{\bf fp}(p,\theta)}(z)=\frac{p\theta z}{1-\theta z}.
\end{equation}
Therefore ${\bf fp}(p,\theta)$ is free infinitely divisible. This distribution appears as the limit of eigenvalue distributions of Wishart matrices as the size of the random matrices
goes to infinity. 

\end{example}

\subsection{Univalent Inverse Cauchy transform}

In section 2.2, we introduced the free infinitely divisible distributions. Free infinite divisibility for probability measures is equivalent to that its free cumulant transform has the free L\'{e}vy-Khintchine representation \eqref{LKrep} (or its Voiculescu transform has the Pick-Nevanlinna representation). However, in general, it is very difficult to find the free characteristic triplet (or free generating pair) of probability measures to check free infinite divisibility for probability measures. Many people studied conditions of free infinite divisibility.

In particular, Arizmendi and Hasebe (\cite[Definition 5.1]{AH13}) defined a special class of probability measures as a criterion for free infinite divisibility.

\begin{definition}
A probability measure $\mu$ on $\mathbb{R}$ is said to be in the class UI (Univalent Inverse Cauchy transform) if the right inverse function $G_\mu^{-1}$, originally defined in a triangular domain $\Lambda_{\gamma,\delta}$, has univalent analytic continuation to $\mathbb{C}^-$. In particular, if $X\sim \mu$ and $\mu\in$ UI, then we write $X\sim$ UI.
\end{definition}

In \cite[Proposition 5.2 and p.2763]{AH13}, the class UI has the following properties.

\begin{lemma}\label{UI}
$\mu\in$ UI implies that $\mu$ is free infinitely divisible, and the class UI is closed with respect to weak convergence.
\end{lemma}

\begin{example}\label{UIex}
(1) The semicircle distribution $S(m,\sigma^2)$ is in the class UI. The free Poisson distribution ${\bf fp}(p,\theta)$ is in the class UI (see \cite[Section 2.3 (1),(2)]{Has16}).

(2) The beta distribution $\beta_{p,q}$ is the probability measure with pdf
\begin{equation}
\frac{1}{B(p,q)}x^{p-1}(1-x)^{q-1} 1_{(0,1)}(x), \qquad p,q>0.
\end{equation}
The beta distribution $\beta_{p,q}$ is in the class UI if (i) $p,q\ge 3/2$, (ii) $0<p\le 1/2$, $q\ge 3/2$ or (iii) $0<q\le1/2$, $p\ge 3/2$ (see \cite[Theorem 1.2]{Has14} and \cite[Theorem 3.4]{Has16}).
\end{example}

%%%%%%%%%%%%%%%%%%%%%%%%%%%%%%%%%%%%%%%%%%%%%%%%%%%
%Section 3 
%%%%%%%%%%%%%%%%%%%%%%%%%%%%%%%%%%%%%%%%%%%%%%%%%%%

\section{UI property}

In this section, $\tilde{G}_\mu(z)$ denotes the Cauchy transform  of a complex measure $\mu$ on $\mathbb{R}$ defined in \eqref{OriginalCauchy}. In \cite[Proposition 4.1]{Has14}, Hasebe proved that the Cauchy transform has an analytic continuation to a domain $\mathbb{C}^+\cup \text{supp}(\mu)\cup D$, for some subdomain $D$ of $\mathbb{C}^-$ in the case when a pdf of $\mu$ has some analytic properties.
\begin{lemma}\label{Has}
Consider $0<a<b$. Let $(a,b)$ be an interval of $\mathbb{R}$.  Suppose that $f$ is analytic in a neighborhood of $(a,b)\cup \{z\in\mathbb{C}^-: \arg(z)\in(-\theta,0)\}$ for some $0<\theta\le\pi$ and that $f$ is integrable on $(a,b)$ with respect to the Lebesgue measure. We define the complex measure $\sigma(dx):=f(x)1_{(a,b)}(x)dx$. Then the Cauchy transform $G_\sigma:=\tilde{G}_\sigma|_{\mathbb{C}^+}$ defined on $\mathbb{C}^+$ has an analytic continuation to $\mathbb{C}^+\cup(a,b)\cup\{z\in\mathbb{C}^-: \arg(z)\in(-\theta,0)\}$, which we denote by the same symbol $G_\sigma$, and
\begin{equation}\label{Cauchy}
G_\sigma(z)=\tilde{G}_\sigma(z)-2\pi i f(z), \qquad z\in\{z\in\mathbb{C}^-: \arg(z)\in(-\theta,0)\}.
\end{equation}
\end{lemma}

We prove that a probability measure whose pdf satisfies the following properties is in the class UI and therefore the measure is free infinitely divisible.

\begin{theorem}\label{lemma}
Suppose that $0<a<b$. Let $\mu$ be a probability measure on $(a,b)$ which has a pdf $f(x)$ where $f$ is real analytic, positive and there exists $\theta\in (0,\pi]$ such that\\
{\bf (A1)} $f$ analytically extends to a function (also denoted by $f$) defined in $\{z\in\mathbb{C}\setminus\{0\}: \arg (z)\in(-\theta,0)\}\cup(a,b)$. Moreover $f$ extends to a continuous function on $\{z\in\mathbb{C}\setminus \{0\}: \arg(z)\in[-\theta,0]\}$;\\
{\bf (A2)} Re$(f(x-i0))=0$ for all $0<x<a$ and $x>b$; \\
{\bf (A3)} There exist $\alpha>0$ and $0<l<\frac{2\pi-\theta}{\theta}$ such that
\begin{equation}
f(z)=-\frac{i\alpha}{z^{1+l}}(1+o(1)), \qquad \text{as } z\rightarrow0, \hspace{2mm}\arg (z) \in (-\theta,0);
\end{equation}
{\bf (A4)} Re$(f(ue^{-i \theta}))\le 0$ (if $\theta=\pi$, then Re$(f(-u))\le0$) for all $u>0$;\\
{\bf (A5)} $\lim_{|z|\rightarrow \infty, \arg(z)\in(-\theta,0)} f(z)=0$.\\\vspace{-0.3cm}

By Lemma \ref{Has} and our assumption (A1), the Cauchy transform $G(z):=G_\mu(z)$ of $\mu$ has analytic continuation to $\mathbb{C}^+ \cup(a,b) \cup \{z\in\mathbb{C}^-: \arg(z)\in(-\theta,0)\}$, and we denote by the same symbol $G$. Then \\\vspace{-0.3cm}\\
{\bf (A6)} $G(z)$ can extend to a continuous function on $\mathbb{C}^+\cup[a,b]\cup\{z\in\mathbb{C}\setminus\{0\}: \arg(z)\in[-\theta,0)\}\cup((-\infty,a]\cup[b,\infty)+i0)\cup([0,a]\cup[b,\infty)-i0)$.
\\\vspace{-0.4cm}\\
Then $\mu\in$ UI.
\end{theorem}
\begin{proof}
We define the following 8 lines and curves (see Figure \ref{ck}). In the following $\eta>0$ is supposed to be large and $\delta>0$ is supposed to be small.
\begin{itemize}
\item $c_1$ is the real line segment from $-\eta +i0$ to $a+i0$;
\item $c_2$ is the real line segment from $a-i0$ to $\delta-i0$;
\item $c_3$ is the clockwise circle $\delta e^{i\psi}$ where $\psi$ starts from $0$ and ends with $-\theta$;
\item $c_4$ is the line segment from $\delta e^{-i\theta}+0$ to $\eta e^{-i \theta}+0$;
\item  $c_5$ is the counterclockwise circle $\eta e^{i\psi}$ where $\psi$ starts from $-\theta$ and ends with $0$;
\item $c_6$ is the real line segment from $\eta-i0$ to $b-i0$;
\item $c_7$ is the real line segment form $b+i0$ to $\eta+i0$;
\item $c_8$ is the counterclockwise circle $\eta e^{i\psi}$ where $\psi$ starts from $0$ and ends with $\pi$. 
\end{itemize}

\begin{figure}[htbp]
\begin{center}
  \begin{tabular}{c}

 \begin{minipage}{0.45\hsize}
      \centering 
        \includegraphics[clip, width=9.5cm]{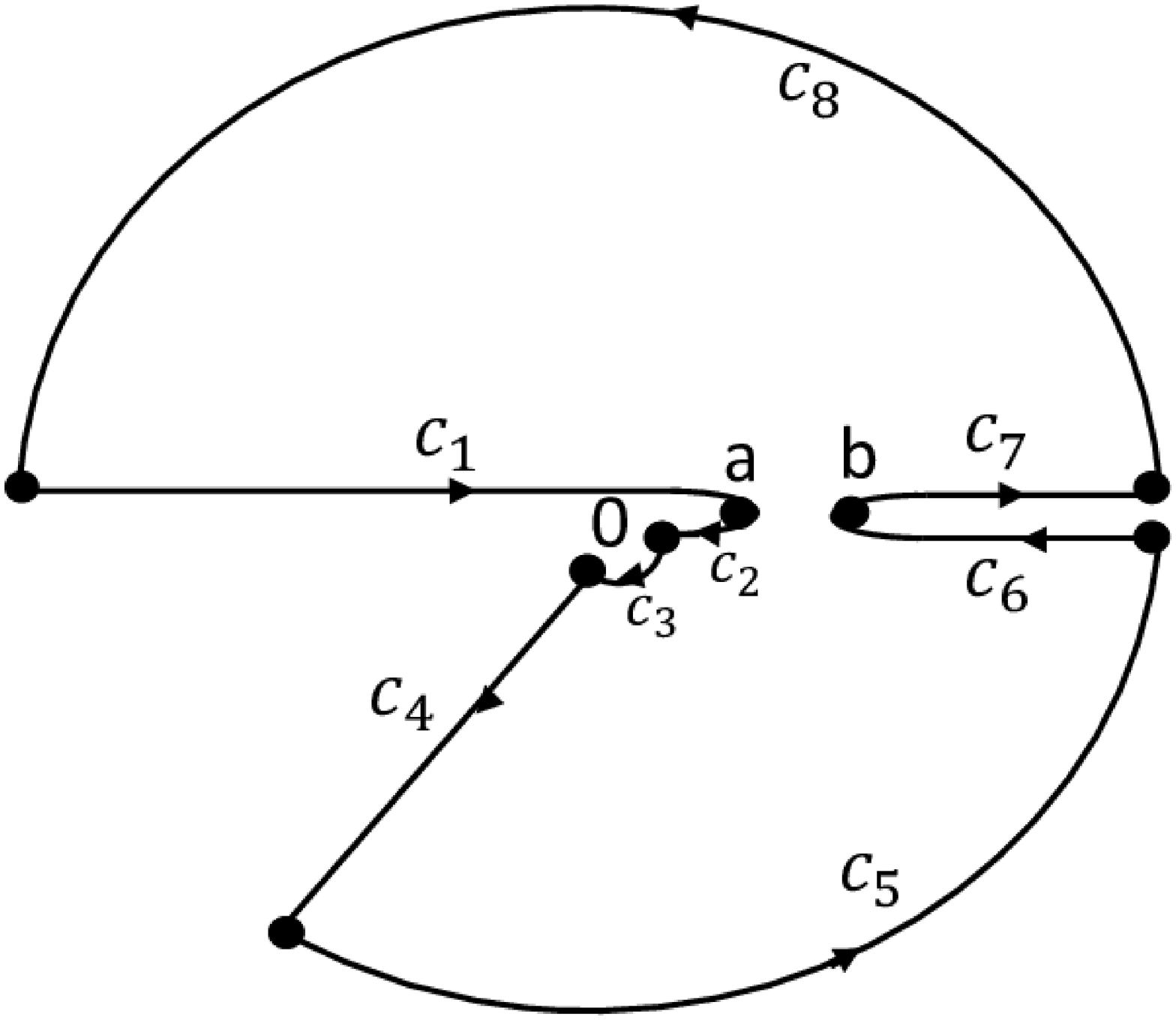}
	\caption{The curves $c_k$}\label{ck}
    \end{minipage}

    \begin{minipage}{0.5\hsize}
      \centering 
        \includegraphics[clip, width=9.5cm]{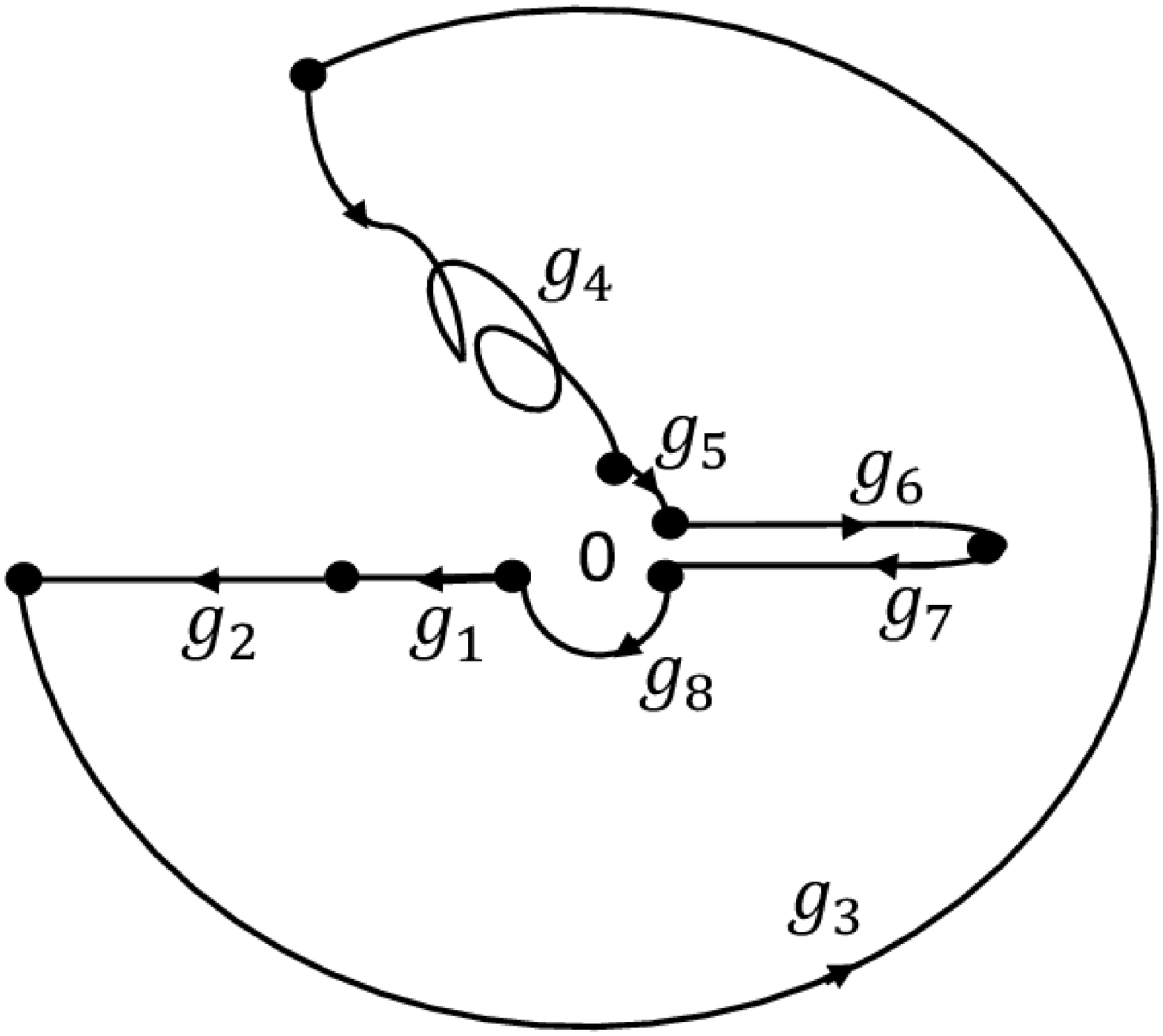}
	\caption{The curves $g_k$}\label{gk}
    \end{minipage}

  \end{tabular}
  \label{fig:img}
  \end{center}
\end{figure}

Define the curves $g_k:=G(c_k)$ for all $k=1,\dots ,8$ (see Figure \ref{gk}). It is easy to prove that $g_1$ is the negative real line, $g_7$ is the positive real line and $g_8$ is the small clockwise circle whose argument starts from 0 and ends with $-\pi$. By our assumption (A6), the curves $g_1\cup g_2$ and $g_6\cup g_7$ are continuous since so are $c_1\cup c_2$ and $c_6\cup c_7$. By our assumption (A2), we have that Re$(f(x-i0))=0$ for all $0<x<a$ and $x>b$. Hence Im$(G(x-i0))=0$ for all $0<x<a,$ $x>b$, and therefore $g_2$ and $g_6$ are real lines. By our assumption (A4), we have 
\begin{equation}
\text{Im}(G(ue^{-i\theta}))=\text{Im}(\tilde{G}(ue^{-i\theta}))-2\pi \text{Re}(f(ue^{-i\theta}))>0, \qquad u>0.
\end{equation}
Therefore $g_4$ is contained in $\mathbb{C}^+$. Let $\epsilon>0$ be supposed to be small. Since our assumption (A3) and the condition \eqref{Cauchy} hold and $\tilde{G}(z)=o(1/z)$ as $z\rightarrow 0$, $z\in\mathbb{C}^-$, we have that
\begin{equation}\label{asympt}
G(z)=-\frac{2\pi \alpha}{z^{1+l}}(1+o(1)), \qquad \text{as } z\rightarrow 0, \hspace{2mm} \arg(z)\in (-\theta,0).
\end{equation}
By asymptotics \eqref{asympt}, we can take $\delta\in(0, (\pi\alpha\epsilon)^{\frac{1}{1+l}})$ small enough such that
\begin{equation}
|G(\delta e^{i \psi})+2\pi \alpha (\delta e^{i\psi})^{-(1+l)}|< \pi\alpha\delta^{-(1+l)},
\end{equation}
uniformly on $\psi\in(-\theta,0)$. Hence we have
\begin{equation}
|G(\delta e^{i\psi})|>\pi\alpha\delta^{-(1+l)}>\frac{1}{\epsilon},
\end{equation}
and therefore the distance between the curve $g_3$ and $0$ is larger than $\frac{1}{\epsilon}$.  Since we have that $G(\delta e^{i\psi})=\frac{2\pi\alpha}{\delta^{1+l}}e^{(-\pi-\psi(1+l))i}$, and the argument $\psi$ starts from $0$ and ends with $-\theta$, the $g_3$ is the large counterclockwise circle whose argument starts from $-\pi$ and ends with $-\pi+\theta(1+l)$ ($<\pi$ by (A3)). 

Note that $\tilde{G}(z)=\frac{1}{z}(1+o(1))$ as $|z|\rightarrow \infty$, $z\in\mathbb{C}^-$. By our assumption (A5), we have that $|f(z)|\rightarrow 0$ as $|z|\rightarrow \infty$, $\arg(z)\in(-\theta,0)$, and therefore
\begin{equation}
|G(z)|\le |\tilde{G}(z)|+2\pi|f(z)|\rightarrow 0, 
\end{equation}
as $|z|\rightarrow\infty$, $\arg(z)\in(-\theta,0)$. Hence there exists $\eta>0$ large enough such that $|z|\ge \eta$ implies that $|G(z)|<\epsilon$, that is, $g_5$ is contained in a small circle with radius $\epsilon>0$. 

Therefore every point of $D_{\epsilon}:=\{ z\in\mathbb{C}^-: \epsilon< |z| < \frac{1}{\epsilon}\}$ is surrounded by the closed curve $g_1\cup \dots \cup g_8$ exactly once. By the argument principle, for any $w\in D_{\epsilon}$, there exists only one element $z$ in the bounded domain surrounded by the closed curve $c_1\cup \dots \cup c_8$, such that $G(z)=w$. Therefore we can define a right inverse function $G^{-1}$ on $D_{\epsilon}$. Since $G^{-1}(D_{\epsilon})$ is a connected domain and $G$ is univalent on the domain $G^{-1}(D_{\epsilon})$, the inverse function $G^{-1}$ is univalent on $D_{\epsilon}$. We can define a univalent right inverse function (as the same symbol $G^{-1}$) in $\mathbb{C}^{-}$  by letting $\epsilon \rightarrow 0$. By the identity theorem, the right inverse function, originally defined in some triangular domain $\Lambda_{\gamma,\delta}$ has a univalent analytic continuation to our function $G^{-1}$ on $\mathbb{C}^-$. Hence we can conclude that $\mu\in$ UI. 
\end{proof}

%%%%%%%%%%%%%%%%%%%%%%%%%%%%%%%%%%%%%%%%%%%%%%%%%%%
%Section 4 
%%%%%%%%%%%%%%%%%%%%%%%%%%%%%%%%%%%%%%%%%%%%%%%%%%%

\section{Free infinite divisibility for generalized power distributions with free Poisson term}
In \cite{Has16}, if a random variable $X$ follows the free Poisson distribution (Marchenko-Pastur distribution) ${\bf fp}(p)$, then $X^r$ follows an FID (UI) distribution, where $p>0$ and $r\in\mathbb{R}$ have some properties. We prove that if $X$ follows the generalized power distribution with free Poisson term (depend on five parameters $a,b, N,\alpha,l$), then $X^r$ follows an FID (UI) distribution when $l\in [t,t+1]_<^N$ for some $t\ge0$ and $r\ge 1$ (in particular $l\in[0,1]_<^N$, $|r|\ge 1$). The class of these distributions is the generalization of the free Poisson distributions and the free Generalized Inverse Gaussian (fGIG) distributions (See Example \ref{ex}).

\begin{theorem}\label{Thm}
Consider $0<a<b$. Suppose that a random variable $X$ follows the {\it generalized power distributions with free Poisson term} (for short, {\it GPFP distributions}), that is, 
\begin{equation}\label{GPFP}
\begin{split}
X\sim \frac{\sqrt{(b-x)(x-a)}}{x}\sum_{k=1}^N \frac{\alpha_k}{x^{l_k}} 1_{(a,b)}(x)dx&=:f_{a,b,N,\alpha,l}(x)dx\\
&=:\text{GPFP}(a,\,b,\,N,\,\alpha,\,l),
\end{split}
\end{equation}
where $N\in\mathbb{N}$, $\alpha=(\alpha_1,\dots ,\alpha_N)\in(0,\infty)^N$ with $\int_a^b f_{a,b,N,\alpha,l}(x)dx=1$ and $l=(l_1,\dots ,l_N)\in\mathbb{R}_{<}^N$. If $l\in [t,t+1]_<^N$ for some $t\ge0$, then $X^r\sim$ UI if $r\ge 1$.
\end{theorem}
\begin{proof}
Fix a number $N\in\mathbb{N}$. Consider first that $r>1$ and $l=(l_1,\dots, l_N)\in[t,t+1]_<^N$ satisfies $t<l_1<\dots <l_N<t+1$ for some $t\ge0$. Assume that $[t]=n-1$ for some $n\in\mathbb{N}$. Define $\tilde{l}_k:=l_k-[t]=l_k-(n-1)\in(0,1)$ for all $k=1,\dots ,N$. Then we have
\begin{equation}
X^r\sim \frac{s\sqrt{(B^s-x^s)(x^s-A^s)}}{x} \sum_{k=1}^N \frac{\alpha_k}{x^{s(n-1+\tilde{l}_k)}}\cdot 1_{(A,B)}(x)dx=:h(x)dx,
\end{equation}
where $s=\frac{1}{r}\in(0,1)$, $A=a^r$ and $B=b^r$. We define a function $h_k(x)$ as the $k$-th term of $h(x)$. Set $n_0:=s(n-1)+1\ge1$. Note that $s(n-1+\tilde{l}_k)=n_0-1+\tilde{l}_ks$ for each $k=1,\dots ,N$. Consider $\theta=\frac{\pi}{n_0}$. Let $G=G_{X^r}$ be the Cauchy transform of $X^r$. We prove that $\theta$, $h$ and $G$ satisfy the assumptions from (A1) to (A6) in Theorem \ref{lemma}. Suppose that $z\in\mathbb{C}\setminus \{0\}$ and $\arg(z)\in(-\theta,0)$. By a similar proof of \cite[Theorem 3.5]{Has16}, we have that $\arg((B^s-z^s)(z^s-A^s))\in(-\pi,\pi)$. Moreover we also have that $\arg(z^{s(n-1+\tilde{l}_k)})\in (-\theta s(n-1+\tilde{l}_k),0)\subset (-\pi,0)$. Thus every $h_k$ analytically extends to a function defined in $\{z\in\mathbb{C}\setminus\{0\}: \arg(z)\in(-\theta,0)\}\cup (A,B)$, so that the assumption (A1) holds. For $0<x<A$ we have
\begin{equation}
h_k(x-i0)=-\frac{s\alpha_k\sqrt{(A^s-x^s)(B^s-x^s)}}{x^{1+s(n-1+\tilde{l}_k)}} i.
\end{equation}
For $x>B$ we have
\begin{equation}
h_k(x-i0)=\frac{s\alpha_k\sqrt{(x^s-A^s)(x^s-B^s)}}{x^{1+s(n-1+\tilde{l}_k)}} i.
\end{equation}
Therefore Re$(h_k(x-i0))=0$ for all $0<x<A$ and $x>B$, so that the assumption (A2) holds. For each $k=1,\cdots, N$, we have
\begin{equation}
h_k(z)=-\frac{is\alpha_k \sqrt{(AB)^s}}{z^{1+s(n-1+\tilde{l}_k)}}(1+o(1)), \qquad \text{as } z\rightarrow0, \arg(z)\in\left(-\theta,0\right). 
\end{equation}
Moreover we have that $0<s(n-1+\tilde{l}_k)<\frac{2\pi-\theta}{\theta}$. Therefore the assumption (A3) holds. Consider first $0<\theta<\pi$ (equivalently, $n>1$). For all $k=1,\dots ,N$, $u>0$ we have
\begin{equation}
\begin{split}
\arg(h_k(ue^{-i\theta}))&=\arg \left( \frac{s\alpha_k \sqrt{(B^s-u^se^{-is\theta})(u^se^{-is\theta}-A^s)}}{u^{1+s(n-1+\tilde{l}_k)}e^{-i\theta(1+s(n-1+\tilde{l}_k))}} \right)\\
&\in \left(-\frac{\pi}{2}+\theta(1+s(n-1+\tilde{l}_k)),\frac{\pi-2s\theta}{2}+\theta(1+s(n-1+\tilde{l}_k))\right)\\
&= \left(-\frac{\pi}{2}+\theta(n_0+s\tilde{l}_k),\frac{\pi-2s\theta}{2}+\theta(n_0+s\tilde{l}_k)\right)\\
&=\left(\frac{\pi}{2}+\theta s\tilde{l}_k, \frac{3}{2}\pi-\theta s(1-\tilde{l}_k)\right)\\
& \subset \left(\frac{\pi}{2},\frac{3}{2}\pi \right),
\end{split}
\end{equation}
since we have
\begin{equation}
\arg\left(\sqrt{(B^s-u^se^{-is\theta})(u^se^{-is\theta}-A^s)}\right) \in \left(-\frac{\pi}{2}, \frac{\pi-2s\theta}{2}\right).
\end{equation}
Therefore Re$(h_k(ue^{-i\theta}))\le0$ for all $u>0$. When $\theta=\pi$ (equivalently, $n=1$), for all $k=1,\dots ,N$, $u<0$ we have that
\begin{equation}
\begin{split}
\arg(h_k(u))&=\arg\left(\frac{s\alpha_k \sqrt{(B^s-(u-i0)^s)((u-i0)^s-A^s)}  }{u^{1+s(n-1+\tilde{l}_k)}}\right)\\
&=\arg\left(\frac{s\alpha_k \sqrt{(B^s-(u-i0)^s)((u-i0)^s-A^s)}  }{|u|^{1+s\tilde{l}_k}e^{-i\pi(1+s\tilde{l}_k)}}\right)\\
& \in \left(-\frac{\pi}{2}+\pi(1+s\tilde{l}_k), \frac{1-2s}{2}\pi+ \pi(1+s\tilde{l}_k) \right)\\
& = \left(\frac{\pi}{2}+\pi s\tilde{l}_k, \frac{3}{2}\pi- \pi s(1-\tilde{l}_k) \right)\\
&\subset \left(\frac{\pi}{2},\frac{3}{2}\pi\right),
\end{split}
\end{equation}
since
\begin{equation}
\arg\left( \sqrt{(B^s-(u-i0)^s)((u-i0)^s-A^s)} \right)\in \left(-\frac{\pi}{2}, \frac{1-2s}{2}\pi\right).
\end{equation}
Therefore Re$(h_k(u))\le 0$ for all $u<0$. Hence we can conclude that the assumption (A4) holds. Finally, we have
\begin{equation}
h_k(z)\sim \frac{-is\alpha_k}{z^{1+s(n-2+\tilde{l}_k)}}, \text{ as } |z|\rightarrow \infty, \arg(z)\in \left(-\theta,0\right).
\end{equation}
Hence $\lim_{|z|\rightarrow\infty,\arg(z)\in(-\theta,0)}h_k(z)=0$, so that the assumption (A5) holds. Finally, we can conclude that the function $h$ satisfies from (A1) to (A5) since so is every function $f_k$, $(k=1,\dots ,N)$. By taking $\alpha=\frac{3}{2}$, $x_0=A$ and $x_0=B$ in \cite[Theorem 5.1 (5.6)]{Has14}, the Cauchy transform $G$ satisfies the assumption (A6). By Theorem \ref{lemma}, we have $X^r\sim$ UI.

By the w-closedness of UI (see Lemma \ref{UI}), we have that $X^r\sim$ UI even if $r\ge 1$ and $l\in[t,t+1]_<^N$ satisfies that $t\le l_1<\dots <l_N\le t+1$ for $t\ge0$.
\end{proof}

\begin{corollary}\label{Cor}
If $X\sim \text{GPFP}(a,b,N,\alpha,l)$ and $l\in [0,1]_<^N$ then $X^r\sim$ UI when $r\le -1$.
\end{corollary}
\begin{proof}
Suppose that $r\le -1$. We have
\begin{equation}
X^r \sim \frac{t\sqrt{(AB)^{-t}}\sqrt{(B^t-x^t)(x^t-A^t)}}{x} \sum_{k=1}^N \frac{\alpha_k}{x^{t(1-l_k)}}\cdot 1_{(A,B)}(x)dx:=p(x)dx,
\end{equation}
where $t:=-\frac{1}{r}\in(0,1)$, $A:=b^{-1/t}$ and $B:=a^{-1/t}$. We set $\theta=\pi$ and a function $p_k(x)$ as the $k$-th term of $p(x)$. By a similar proof of Theorem \ref{Thm}, we have that $\theta$, $p_k$, $p$, and $G=G_{X^r}$ satisfy the assumptions from (A1) to (A6) in Theorem \ref{lemma}.  Therefore $X^r\sim$ UI if $r\le -1$. 
\end{proof}

Finally, we have that $X\sim \text{GPFP}(a,b,N,\alpha,l)$ and $l\in[0,1]_<^N$ implies that $X^r\sim$ UI if $|r|\ge 1$ by Theorem \ref{Thm} and Corollary \ref{Cor}, (but we will find a counterexample in the case when $X\sim \text{GPFP}(a,b,N,\alpha,l)$ and $l\in[t,t+1]_<^N$ for some $t>0$ by calculating free cumulants and Hankel determinants in section 5). 

\begin{remark}\label{remark}
Consider a random variable $X$ which satisfies that $X\sim \text{GPFP}(a,\, b,\, N,\, \alpha,\, l)$ where $\alpha=(\alpha_1,\dots , \alpha_N)$ and $l=(l_1, \dots , l_N)$. By a similar calculation in Corollary \ref{Cor}, we have that
\begin{equation}
X^{-1}\sim \text{GPFP}\left( \frac{1}{b},\, \frac{1}{a},\, N, \, (\alpha_N\sqrt{ab},\dots , \alpha_1\sqrt{ab}), \, (1-l_N,\, \dots ,1-l_1) \right).
\end{equation}
\end{remark}

We give five examples of GPFP distributions.

\begin{example}\label{ex}
(1) We put $p>1$, $a=(\sqrt{p}-1)^2$ and $b=(\sqrt{p}+1)^2$, $N=1$, $\alpha=\frac{1}{2\pi}$, $l=0$. Then we have that
\begin{equation}
\begin{split}
\text{GPFP}&\left(a=(\sqrt{p}-1)^2,\, b=(\sqrt{p}+1)^2,\, N=1,\, \alpha=\frac{1}{2\pi},\, l=0\right)\\
&=\frac{\sqrt{\left((\sqrt{p}+1)^2-x\right)\left(x-(\sqrt{p}-1)^2\right)}}{2\pi x}\cdot 1_{((\sqrt{p}-1)^2,(\sqrt{p}+1)^2)}(x)dx\\
&={\bf fp}(p).
\end{split}
\end{equation}
In this case, the GPFP distribution corresponds to the free Poisson distribution (Marchenko-Pastur distribution). By Theorem \ref{Thm} and Corollary \ref{Cor}, we have that $X\sim {\bf fp}(p)$ implies that $X^r\sim$ UI if $|r|\ge1$, (but the result in \cite{Has16} is stronger than this one).\\\\
(2) Consider $n\in\mathbb{N}$, $b>0$. Define the following constant number:
\begin{equation}
c(n,b):=\left(\int_{\frac{1}{n}}^b \frac{\sqrt{(x-1/n)(b-x)}}{2\pi x}dx\right)^{-1}.
\end{equation}
Note that $\lim_{n\rightarrow\infty}c(n,b)=\frac{4}{b}$ for each $b>0$. By Theorem \ref{Thm} and Corollary \ref{Cor}, for all $n\in\mathbb{N}$ we have that
\begin{equation}
\begin{split}
S_{n,b}\sim\text{GPFP}&\left(\frac{1}{b},\, n,\, N=1,\, \alpha=\frac{c(n,b)}{2\pi }\sqrt{\frac{b}{n}},\, l=1\right)\\
&=c(n,b)\sqrt{\frac{b}{n}}\cdot\frac{\sqrt{(x-1/b)\left(n-x\right)}}{2\pi x^2}\cdot1_{(1/b,n)}(x)dx,
\end{split}
\end{equation}
implies that $S_{n,b}^r\sim$ UI for all $|r|\ge 1$. Let $S_b$ be a random variable such that
\begin{equation}\label{freestable}
S_b\sim \frac{4}{b}\cdot\frac{\sqrt{bx-1}}{2\pi x^2} \cdot1_{(1/b,\infty)}(x)dx.
\end{equation}
In particular, if $b=4$, then the probability measure in \eqref{freestable} corresponds to the free positive stable law ${\bf fs}_{1/2}$ with index $1/2$. Therefore the w-closedness of UI implies that $S_b^r\sim$ UI for all $r\ge1$.  By a similar proof, we have that $S_b^r\sim$ UI for all $r\le -1$. In particular, we have that $S_4^{-1}\sim {\bf fp}(1)$.\\\\
(3) Consider $\lambda\in\mathbb{R}$. We put $N=2$, $\alpha=\left(\frac{\alpha_1}{2\pi}, \frac{\alpha_2}{2\pi\sqrt{ab}}\right)$ and $l=(0,1)$, where $0<a<b$. are the solution of
\begin{equation}
\begin{cases}
1-\lambda+\alpha_1\sqrt{ab}-\alpha_2\frac{a+b}{2ab}=0 \\
1+\lambda+\frac{\alpha_2}{\sqrt{ab}}-\alpha_1\frac{a+b}{2}=0.
\end{cases}
\end{equation}
Then we have
\begin{equation}
\begin{split}
\text{GPFP}&\left(a,\, b,\, N=2,\, \alpha=\left(\frac{\alpha_1}{2\pi},\frac{\alpha_2}{2\pi\sqrt{ab}}\right),\, l=(0,1)\right)\\
&=\frac{\sqrt{(x-a)(b-x)}}{2\pi x}\left(\alpha_1+\frac{\alpha_2}{\sqrt{ab}x} \right)\cdot 1_{(a,b)}(x)dx=:\text{fGIG}(\alpha_1,\alpha_2,\lambda).
\end{split}
\end{equation}
In this case, the GPFP distribution corresponds to the free Generalized Inverse Gaussian (fGIG) distribution. This distribution was studied by \cite{HS17} (free version of the generalized inverse gaussian distributions, free selfdecomposability, free regularity and unimodality) and \cite{S17} (free version of Matsumoto-Yor property and R-transform of fGIG distribution). Moreover $X\sim \text{fGIG}(\alpha_1,\alpha_2,\lambda)$ implies that $X^r \sim$ UI if $|r|\ge 1$ by Theorem \ref{Thm} and Corollary \ref{Cor}.\\\\
(4) Let $0<a<b$. We put $N=1$ and $l=-1$. Then we have 
\begin{equation}\label{semicircle}
\begin{split}
\text{GPFP}\left(a,\, b,\, N=1, \, \alpha=\frac{C}{2\pi}, \, l=-1\right)=\frac{C}{2\pi}\sqrt{\left(\frac{b-a}{2}\right)^2-\left(x-\frac{a+b}{2}\right)^2}\cdot 1_{(a,b)}(x)dx,
\end{split}
\end{equation}
where $C>0$ is a constant such that the function \eqref{semicircle} becomes a pdf. If a random variable $X$ follows the probability measure \eqref{semicircle}, then Remark \ref{remark} yields that
\begin{equation}
X^{-1}\sim \text{GPFP}\left(\frac{1}{b},\, \frac{1}{a},\,  1,\, \frac{C}{2\pi}\sqrt{ab}, \, 2 \right).
\end{equation}
By Theorem \ref{Thm}, we have that $X^{-r}=(X^{-1})^{r}\sim$ UI for $r\ge1$. Suppose that $S\sim S(0,1)$ and $u>2$. Then we have that
\begin{equation}
S+u\sim\text{GPFP}\left(a=u-2,\, b=u+2,\, N=1,\, \alpha=\frac{1}{2\pi},\, l=-1 \right).
\end{equation}
For the above reason, we have that $(S+u)^s\sim$ UI for all $s\le -1$. By the w-closedness of UI, we can also conclude that $(S+2)^r\sim$ UI for all $r\le -1$. However, since $(S+u)^2$ does not follow an FID distribution for $u\neq 0$ from the paper \cite{E12}, we can conclude that if $X\sim \text{GPFP}(a=u-2,\,b=u+2,\,N=1,\, \alpha=\frac{1}{2\pi} ,\ l=-1)$ then $X^2$ does not an FID distribution for all $u>2$.\\\\
(5) Consider $n\in\mathbb{N}$ and $l<1/2$. Define the following constant number:
\begin{equation}
\alpha(n,l):=\left(\frac{1}{B\left(\frac{1}{2}-l, \frac{3}{2}\right)}\int_{\frac{1}{n}}^1 \frac{\sqrt{(1-x)(x-1/n)}}{x^{1+l}}dx \right)^{-1}.
\end{equation}
Note that $\lim_{n\rightarrow\infty}\alpha(n,l)=1$ for each $l<1/2$. For all $n\in\mathbb{N}$ and $l<1/2$, we consider that
\begin{equation}
B_{n,l}\sim \text{GPFP}\left(\frac{1}{n}, \, 1, \, N=1, \, \frac{\alpha(n,l)}{B(1/2-l, 3/2)}, \, l \right).
\end{equation}
By Remark \ref{remark}, we have that
\begin{equation}\label{betainverse}
B_{n,l}^{-1} \sim \text{GPFP} \left(1,\, n,\, 1, \, \sqrt{\frac{1}{n}}\cdot\frac{\alpha(n,l)}{B(1/2-l, 3/2)}, \, 1-l\right).
\end{equation}
By Theorem \ref{Thm} and Corollary \ref{Cor}, we have that $(i)$ if $0\le l<1/2$ then $B_{n,l}^r \sim$ UI for $|r|\ge 1$, and $(ii)$ if $l<0$ then the condition \eqref{betainverse} implies that $B_{n,l}^{-r}=(B_{n,r}^{-1})^r \sim$ UI for $r\ge1$. For each $l<1/2$, we consider a random variable $B_l$ which follows the beta distribution $\beta_{1/2-l, 3/2}$. Recall that the beta distributions $\beta_{p,q}$ are defined in Example \ref{UIex} (2). By the w-closedness of UI, we have that $(i)$ if $0\le l <1/2$ then $B_l^r\sim$ UI for $|r|\ge 1$, and $(ii)$ if $l<0$ then $B_l^r \sim$ UI for $r\le -1$. 
\end{example}

%%%%%%%%%%%%%%%%%%%%%%%%%%%%%%%%%%%%%%%%%%%%%%%%%%%
%Section 5　
%%%%%%%%%%%%%%%%%%%%%%%%%%%%%%%%%%%%%%%%%%%%%%%%%%%

\section{Non free infinite divisibility for GPFP distributions}
Let $(\mathcal{A},\phi)$ be a $C^\ast$-probability space and $X\in\mathcal{A}$ a random variable. The {\it $n$-th moment} $m_n=m_n(X)$ of $X$ is defined by the value $\phi(X^n)$. In particular, if $X\sim \mu$, we write $m_n(\mu)$ as the $n$-th moment of $X$ (or $\mu$). The {\it $n$-th free cumulant} $\kappa_n=\kappa_n(X)$ of $X$ is defined as the $n$-th coefficient of power series of the free cumulant transform $C_X(z)$. If $X\sim \mu$, we write $\kappa_n(\mu)$ as the $n$-th free cumulant of $X$ (or $\mu$). We have {\it moment-cumulant formula} as follows
\begin{equation}
\kappa_n=\sum_{\pi \in NC(n)}  \left(\prod_{V\in \pi} m_{|V|}\right) \mu(\pi,1_n), \qquad n\in \mathbb{N},
\end{equation}
where $NC(n)$ is the set of all non-crossing partitions of the finite set $\{1,\dots , n\}$, the symbol $1_n$ is the non-crossing partition which has the block $\{1,\dots, n\}$, the function $\mu(\pi,\sigma)$ $(\pi\le\sigma \text{ in } NC(n))$ is the M\"{o}bius function of $NC(n)$, and $|V|$ is the number of elements in a block $V$ of $\pi$ (for detail, see \cite{NS06}). By the moment-cumulant formula, we can compute free cumulants by using the moments. For example we have that
\begin{equation}\label{cumulant}
\begin{split}
\kappa_1&=m_1,\\
\kappa_2&=m_2-m_1^2,\\
\kappa_3&=m_3-3m_1m_2+2m_1^3,\\
\kappa_4&=m_4-4m_1m_3-2m_2^2+10m_1^2m_2-5m_1^4,\\
\kappa_5&=m_5-5m_1m_4-5m_2m_3+15m_1^2m_3+15m_1m_2^2-35m_1^3m_2+14m_1^5.
\end{split}
\end{equation}
The $n$-th free cumulant of ${\bf fp}(p)$ is given by $\kappa_n({\bf fp}(p))=p$ for all $n\in\mathbb{N}$ (see Example \ref{cumulanttrans} (2)). Then the formula \eqref{cumulant} yields that
\begin{equation}\label{moments}
\begin{split}
m_1({\bf fp}(p))&=p,\\
m_2({\bf fp}(p))&=p(p+1),\\
m_3({\bf fp}(p))&=p(p^2+3p+1),\\
m_4({\bf fp}(p))&=p(p^3+6p^2+6p+1),\\
m_5({\bf fp}(p))&=p(p^4+10p^3+20p^2+10p+1).
\end{split}
\end{equation}
We define moments with degree of a complex number. In particular, we need to study the moments of ${\bf fp}(p)$ with degree of a complex number:
\begin{equation}\label{msfp}
m_s({\bf fp}(p))=\int_0^\infty x^s {\bf fp}(p)(dx), \qquad s\in\mathbb{C},\, p>1.
\end{equation}
This is analytic as a function of $s\in\mathbb{C}$.
 
\begin{lemma}
For $s\in\mathbb{C}$ and $p>1$, we have that
\begin{equation}\label{mfp}
m_s({\bf fp}(p))=\frac{m_{-s-1}({\bf fp}(p))}{(p-1)^{-1-2s}}.
\end{equation}
\end{lemma}
\begin{proof}
Consider $s\in\mathbb{C}$ and $p>1$.  Then we have that
\begin{equation}
\begin{split}
m_s({\bf fp}(p))&=\int_{(\sqrt{p}-1)^2}^{(\sqrt{p}+1)^2} x^{s-1} \frac{\sqrt{\left((\sqrt{p}+1)^2-x\right)\left(x-(\sqrt{p}-1)^2\right)}}{2\pi}dx\\
&=(p-1)\int_{\frac{1}{(\sqrt{p}+1)^2}}^{\frac{1}{(\sqrt{p}-1)^2}} x^{-s-1}\frac{\sqrt{\left(\frac{1}{(\sqrt{p}-1)^2}-x\right)\left(x-\frac{1}{(\sqrt{p}+1)^2}\right)}}{2\pi x}dx\\
&=\frac{1}{(p-1)^{-1-2s}}\int_{(\sqrt{p}-1)^2}^{(\sqrt{p}+1)^2} x^{-s-1}\frac{\sqrt{\left((\sqrt{p}+1)^2-x\right)\left(x-(\sqrt{p}-1)^2\right)}}{2\pi}dx\\
&=\frac{m_{-s-1}({\bf fp}(p))}{(p-1)^{-1-2s}},
\end{split}
\end{equation}
where the second equality holds by changing the variables from $x$ to $x^{-1}$ and the third equality holds by changing the variables from $x$ to $\frac{x}{(p-1)^2}$.
\end{proof}

By Lemma \ref{msfp}, we have that for all $p>1$,
\begin{equation}\label{-1-2moment}
\begin{split}
m_{-1}({\bf fp}(p))&=\frac{m_0({\bf fp}(p))}{p-1}=\frac{1}{p-1},\\
m_{-2}({\bf fp}(p))&=\frac{m_1({\bf fp}(p))}{(p-1)^3}=\frac{p}{(p-1)^3}.
\end{split}
\end{equation}

We study a condition for the following measures to become a GPFP distribution.

\begin{lemma}\label{prob}
Let $p>1$ and $\alpha_1,\alpha_2>0$. \\
(1) The following measure
\begin{equation}\label{2}
\frac{\sqrt{\left(x-\frac{1}{(\sqrt{p}+1)^2}\right)\left(\frac{1}{(\sqrt{p}-1)^2}-x\right)}}{2\pi x^2}\left(\alpha_1+\frac{\alpha_2}{x}\right) 1_{(1/(\sqrt{p}+1)^2,1/(\sqrt{p}-1)^2)}(x)dx,
\end{equation}
is a probability measure on the positive real line if and only if $\alpha_1+\alpha_2 p=p-1$.\\\\
(2) The following measure
\begin{equation}\label{1}
\frac{\sqrt{\left((\sqrt{p}+1)^2-x\right)\left(x-(\sqrt{p}-1)^2\right)}}{2\pi x}\left(\alpha_1+\frac{\alpha_2}{x^2}\right) 1_{((\sqrt{p}-1)^2,(\sqrt{p}+1)^2)}(x)dx,
\end{equation}
is a probability measure on the positive real line if and only if $\alpha_1+\frac{p}{(p-1)^3}\alpha_2=1$.
\end{lemma}

\begin{proof}
(1) Assume that the measure \eqref{2} is a probability measure. Suppose that $X$ is a random variable which follows the probability measure \eqref{2}. Then we have
\begin{equation}\label{inverse}
X^{-1}\sim \frac{1}{p-1}\frac{\sqrt{\left(x-(\sqrt{p}-1)^2\right)\left((\sqrt{p}+1)^2-x\right)}}{2\pi x}(\alpha_1+\alpha_2 x)1_{((\sqrt{p}-1)^2,(\sqrt{p}+1)^2)}(x)dx.
\end{equation}
By our assumption the measure \label{inverse} is also a probability measure. Equivalently, we have
\begin{equation}
1=\frac{1}{p-1}\left(\alpha_1 m_0({\bf fp}(p))+\alpha_2  m_1({\bf fp}(p)) \right)=\frac{1}{p-1} (\alpha_1+\alpha_2 p).
\end{equation}
Therefore the relation $\alpha_1+\alpha_2 p=p-1$ holds. The inverse implication is clear.\\\\
(2) Let $\mu_{p,\alpha_1,\alpha_2}$ be the measure \eqref{1} on the positive real line. It is a probability measure if and only if
\begin{equation}
\begin{split}
1=\int_{(\sqrt{p}-1)^2}^{(\sqrt{p}+1)^2} \mu_{p,\alpha_1,\alpha_2}(dx)&=\alpha_1 m_0({\bf fp}(p))+\alpha_2 m_{-2}({\bf fp}(p)).
\end{split}
\end{equation} 
The calculation \eqref{-1-2moment} implies our conclusion.
\end{proof}

As one of criteria for non free infinite divisibility, we have the method of Hankel determinants of a sequence of free cumulants. First we find an example of parameters such that $X\sim \text{GPFP}(a,b,N,\alpha,l)$, $l\in[t,t+1]_<^N$ for some $t>0$ and $X^{-1}$ does not follow an FID distribution. Consider the following probability measure
\begin{equation}
\begin{split}
\sigma_{\alpha_1,\alpha_2}:&= \text{GPFP}\left(a=\frac{1}{(\sqrt{2}+1)^2}, \, b=\frac{1}{(\sqrt{2}-1)^2},\, N=2,\, \alpha=\left(\frac{\alpha_1}{2\pi},\frac{\alpha_2}{2\pi}\right),\, l=(1,2) \right)\\
&=\frac{\sqrt{\left(x-\frac{1}{(\sqrt{2}+1)^2}\right)\left(\frac{1}{(\sqrt{2}-1)^2}-x\right)}}{2\pi x}\left(\frac{\alpha_1}{x}+\frac{\alpha_2}{x^2}\right) 1_{(1/(\sqrt{2}+1)^2,1/(\sqrt{2}-1)^2)}(x)dx,
\end{split}
\end{equation}
where $\alpha_1,\alpha_2>0$ satisfy the relation $\alpha_1+2\alpha_2=1$ ($0<\alpha_2<1/2$) by Lemma \ref{prob} (1).

\begin{theorem}
Consider $(\alpha_1,\alpha_2)\in \{(x,y)\in(0,\infty)^2: x+2y=1\}$ and a noncommutative random variable $X$. If $X\sim \sigma_{\alpha_1,\alpha_2}$, then $X^{-1}$ does not follow an FID distribution.
\end{theorem}

\begin{proof}
By Remark \ref{remark}, we have
\begin{equation}
\begin{split}
X^{-1}&\sim \text{GPFP}\left((\sqrt{2}-1)^2,\, (\sqrt{2}+1)^2,\, 2, \, \left(\frac{\alpha_2}{2\pi}, \frac{\alpha_1}{2\pi}\right),\, (-1,0)\right)\\
&=\frac{\sqrt{\left(x-(\sqrt{2}-1)^2\right)\left((\sqrt{2}+1)^2-x\right)}}{2\pi x}(\alpha_1+\alpha_2 x)1_{((\sqrt{2}-1)^2,(\sqrt{2}+1)^2)}(x)dx.
\end{split}
\end{equation}
The relation between $\alpha_1$ and $\alpha_2$ and the above formulas \eqref{moments} yield that 
\begin{equation}\label{momentX1}
\begin{split}
m_1(X^{-1})&=(1-2\alpha_2)m_1({\bf fp}(2))+\alpha_2 m_2({\bf fp}(2))=2(1+\alpha_2),\\
m_2(X^{-1})&=(1-2\alpha_2)m_2({\bf fp}(2))+\alpha_2 m_3({\bf fp}(2))=2(3+5\alpha_2),\\
m_3(X^{-1})&=(1-2\alpha_2)m_3({\bf fp}(2))+\alpha_2 m_4({\bf fp}(2))=2(11+23\alpha_2),\\
m_4(X^{-1})&=(1-2\alpha_2)m_4({\bf fp}(2))+\alpha_2 m_5({\bf fp}(2))=2(45+107\alpha_2).\\
\end{split}
\end{equation}
The formula \eqref{cumulant} and the above calculations \eqref{momentX1} follow that
\begin{equation}
\begin{split}
\kappa_2:=\kappa_2(X^{-1})&=-2(2\alpha_2^2-\alpha_2-1),\\
\kappa_3:=\kappa_3(X^{-1})&=2(8\alpha_2^3-6\alpha_2^2-\alpha_2+1), \\
\kappa_4:=\kappa_4(X^{-1})&=-2(2\alpha_2-1)(20\alpha_2^3-10\alpha_2^2-3\alpha_2+1),
\end{split}
\end{equation}
where $\kappa_n$ is the $n$-th free cumulant of $X^{-1}$. This implies that the 2nd Hankel determinant of $\{\kappa_n\}_{n\ge 2}$:
\begin{equation}
   \left|
    \begin{array}{ccc}
      \kappa_{2} & \kappa_{3}  \\
      \kappa_{3} & \kappa_{4}
    \end{array}
  \right| =\kappa_2\kappa_4-\kappa_3^2,
\end{equation}
is negative for all $0<\alpha_2<1/2$ (see Figure \ref{hankel1}). Therefore $X^{-1}$ does not follow an FID distribution.
\end{proof}

\begin{figure}[htbp]
\begin{center}
  \begin{tabular}{c}

 \begin{minipage}{0.5\hsize}
      \centering 
        \includegraphics[clip, width=8cm]{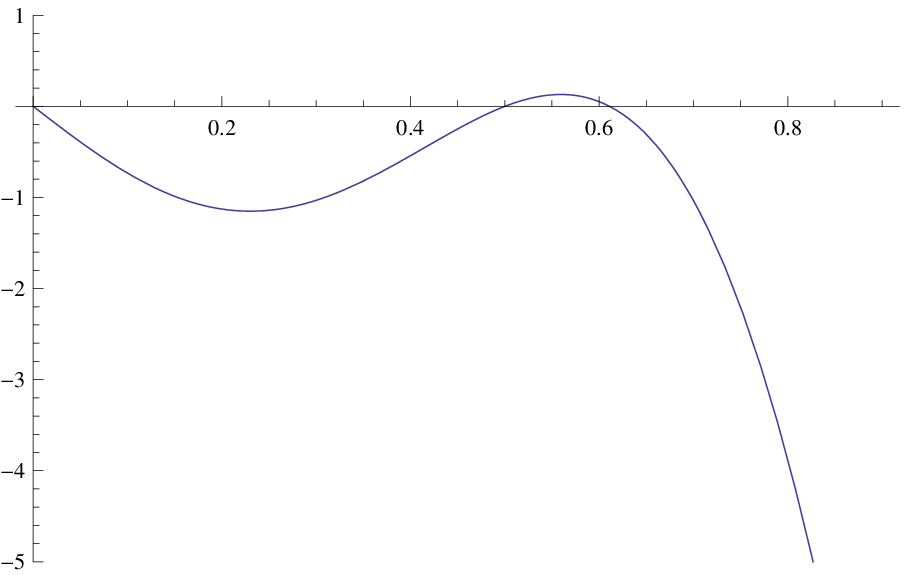}
        \caption{$\kappa_2\kappa_4-(\kappa_3)^2$ as a function of $\alpha_2$. It is negative when $0<\alpha_2<1/2$.}\label{hankel1}
    \end{minipage}

  \end{tabular}
  \label{fig:img}
  \end{center}
\end{figure}

Next we find an example of 5 parameters $(a,\, b,\, N,\, \alpha,\, l)$ with $N\ge 2$ and the width of $l$ is grater than one (in other words, we can find $1< k \le N$ with $l_1<\dots <l_{k-1}\le l_1+1< l_{k}<\dots <l_N$) such that $\text{GPFP}(a,\, b,\, N,\, \alpha,\, l)$ is not free infinitely divisible. Consider a probability measure $\eta=\eta_{\alpha_1,\alpha_2}$ as follows
\begin{equation}
\begin{split}
\eta_{\alpha_1,\alpha_2}:&=\text{GPFP}\left(a=(\sqrt{2}-1)^2,\, b=(\sqrt{2}+1)^2,\, N=2,\, \alpha=\left(\frac{\alpha_1}{2\pi},\frac{\alpha_2}{2\pi}\right),\, l=(0,2)\right)\\
&=\frac{\sqrt{\left((\sqrt{2}+1)^2-x\right)\left(x-(\sqrt{2}-1)^2\right)}}{2\pi x}\left(\alpha_1+\frac{\alpha_2}{x^2}\right)\cdot 1_{((\sqrt{2}-1)^2,(\sqrt{2}+1)^2)}(x)dx,
\end{split}
\end{equation}
where $\alpha_1,\alpha_2>0$ satisfy the relation $\alpha_1+2\alpha_2=1$ by Lemma \ref{prob} (2).

\begin{theorem}\label{counter}
There exists $(\alpha_1,\alpha_2)\in \{(x,y)\in(0,\infty)^2: x+2y=1\}$ such that the probability measure $\eta_{\alpha_1,\alpha_2}$ is not free infinitely divisible.
\end{theorem}
\begin{proof}
The relation between $\alpha_1$ and $\alpha_2$ and the above formulas \eqref{moments} and \eqref{-1-2moment} yield that 
\begin{equation}\label{momentmu}
\begin{split}
m_1(\eta)&=(1-2\alpha_2)m_1({\bf fp}(2))+\alpha_2 m_{-1}({\bf fp}(2))=2-3\alpha_2,\\
m_2(\eta)&=(1-2\alpha_2)m_2({\bf fp}(2))+\alpha_2 m_0({\bf fp}(2))=6-11\alpha_2,\\
m_3(\eta)&=(1-2\alpha_2)m_3({\bf fp}(2))+\alpha_2 m_1({\bf fp}(2))=2(11-21\alpha_2),\\
m_4(\eta)&=(1-2\alpha_2)m_4({\bf fp}(2))+\alpha_2 m_2({\bf fp}(2))=6(15-29\alpha_2).
\end{split}
\end{equation}
The formula \eqref{cumulant} and the above calculations \eqref{momentmu} imply that
\begin{equation}
\begin{split}
\kappa_2':=\kappa_2(\eta)&=-9\alpha_2^2+\alpha_2+2,\\
\kappa_3':=\kappa_3(\eta)&=-54\alpha_2^3+9\alpha_2^2+6\alpha_2+2, \\
\kappa_4':=\kappa_4(\eta)&=-405\alpha_2^4+90\alpha_2^3+34\alpha_2^2+10\alpha_2+2.
\end{split}
\end{equation}
This implies that the 2nd Hankel determinant of $\{\kappa'_n\}_{n\ge 2}$, is negative if $0<\alpha_2\preceq 0.157781$ (see Figure \ref{hankel2}). For example, the probability measure $\eta_{0.7,0.15}$ is not free infinitely divisible.
\end{proof}

\begin{figure}[htbp]
\begin{center}
  \begin{tabular}{c}

 \begin{minipage}{0.45\hsize}
      \centering 
        \includegraphics[clip, width=5.5cm]{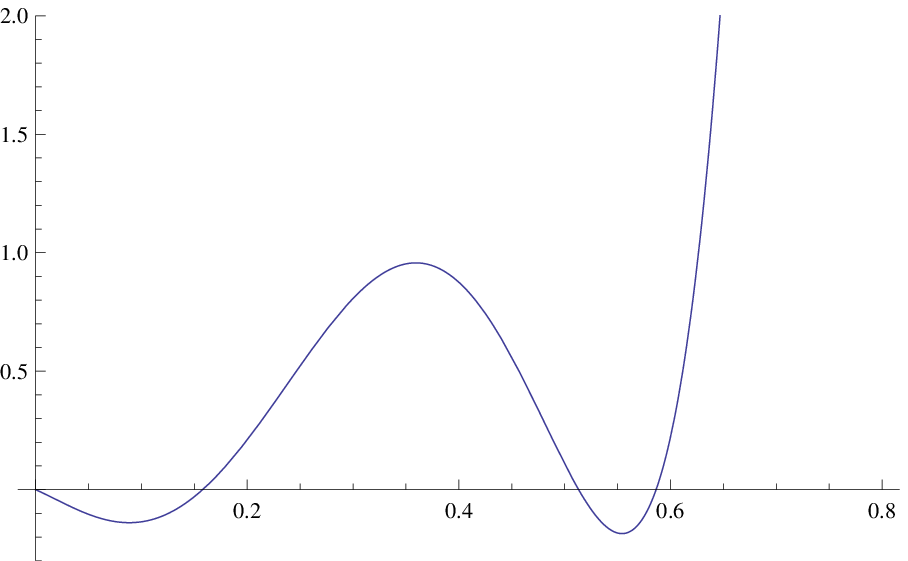}
        \caption{$\kappa'_2\kappa'_4-(\kappa'_3)^2$ as a function of $\alpha_2$. It is negative when $0<\alpha_2\preceq 0.157781$.}\label{hankel2}
    \end{minipage}

    \begin{minipage}{0.4\hsize}
      \centering 
        \includegraphics[clip, width=5.2cm]{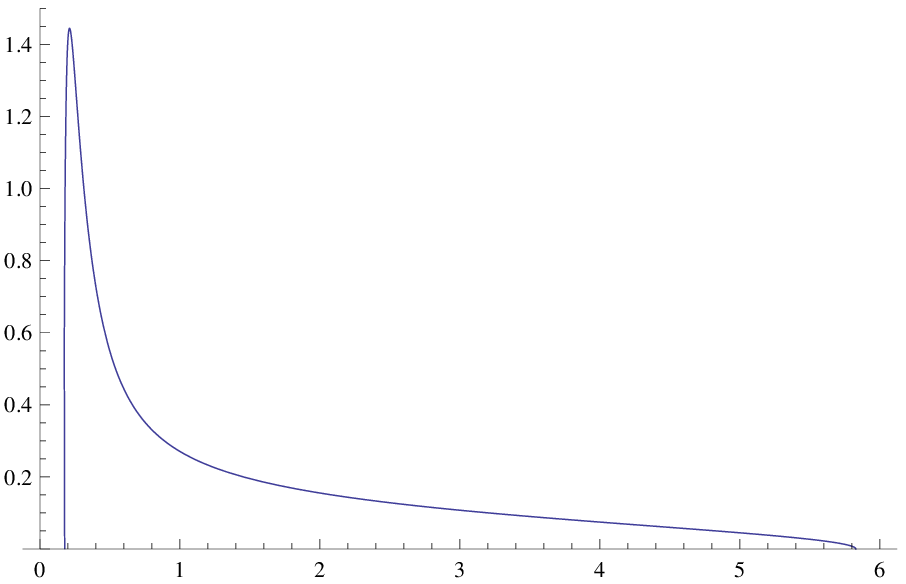}
        \caption{$\eta_{0.7,0.15}$}
    \end{minipage}

  \end{tabular}
  \label{fig:img}
  \end{center}
\end{figure}

\section*{Acknowledgement}
Both authors would like to express his hearty thanks to Professor Takahiro Hasebe who is their advisor in Hokkaido university.  In particular, thanks to his important advices, the authors could advance their researches. This research is supported by JSPS Grant-in-Aid for Young Scientists (B) 15K17549, by Grant-in-Aid for Scientific Research (B) 18H01115 and by JSPS and MAEDI under the Japan-France Integrated Action Program (SAKURA).

\vspace{0.6cm}
{\it \hspace{-5mm}Junki Morishita\\
Department of Mathematics, Hokkaido University,\\
Kita 10, Nishi 8, Kita-Ku, Sapporo, Hokkaido, 060-0810, Japan\\
email: jmorishita@math.sci.hokudai.ac.jp}\\\\
{\it Yuki Ueda\\
Department of Mathematics, Hokkaido University,\\
Kita 10, Nishi 8, Kita-Ku, Sapporo, Hokkaido, 060-0810, Japan\\
email: yuuki1114@math.sci.hokudai.ac.jp}
\end{document}